\documentclass[10pt]{article}
\usepackage{latexsym}
\usepackage{amssymb}
\usepackage{amsthm}
\usepackage{authblk}
\usepackage[english,french]{babel}
\newenvironment{poliabstract}[1]
   {\renewcommand{\abstractname}{#1}\begin{abstract}}
   {\end{abstract}}
\newtheorem{theorem}{Theorem}

\newtheorem{proposition}[theorem]{Proposition}

\date{}

\begin{document}
\title{\small{TRANSCENDENCE OF ZEROS OF AUTOMORPHIC FORMS FOR CUSPIDAL TRIANGLE GROUPS\\
To be published in: Mathematical Reports of the Academy of Science
of the Royal Society of Canada}}
\author{\small{Paula Tretkoff}\footnote{The maiden name of the author is Paula B. Cohen. She authored \cite{Co1} and coauthored \cite{CoWo1}.}}
\maketitle
\selectlanguage{english}
\begin{poliabstract}{Abstract} {We extend some results of \cite{GMR1} on elliptic modular forms. We take \emph{any} Fuchsian triangle group with a cusp and look at power series expansions in a natural parameter around that cusp. Consider the automorphic forms for such a triangle group  whose power series expansions in the natural parameter have algebraic coefficients. We show that the zeros of such forms are either transcendental, or are ``CM." By ``CM," we mean they correspond  to abelian varieties with complex multiplication. This result is the first of its kind in the case of non-arithmetic groups.}\end{poliabstract}
\selectlanguage{french}
\begin{poliabstract}{R\'esum\'e}{Nous \'etendons certains r\'esultats de \cite{GMR1} sur les formes modulaires elliptiques. Nous prenons un groupe fuchsien triangulaire \emph{quelconque} avec une pointe et examinons les d\'eveloppements en s\'erie de puissance dans un param\`etre naturel autour de cette pointe. Consid\'erons les formes automorphes pour un tel groupe triangulaire dont les d\'eveloppements en s\'erie de puissance dans le param\`etre naturel ont des coefficients alg\'ebriques. Nous montrons que les z\'eros de telles formes sont soient transcendants soient ``CM". Par ``CM," nous voulons dire qu'ils correspondent \`a des vari\'et\'es ab\'eliennes \`a multiplication complexe. Ce r\'esultat est le premier du genre au cas des groupes non-arithm\'etiques.}\end{poliabstract}
\selectlanguage{english}

\medskip


\smallskip


\medskip

A point $\tau$ in the upper half plane $\mathcal{H}$ such that $[\mathbb{Q}(\tau):\mathbb{Q}]=2$ is called an imaginary quadratic point of $\mathcal{H}$. It is also called a complex multiplication (CM) point since there are endomorphisms of the lattice $\mathcal{L}_\tau=\mathbb{Z}+\tau\mathbb{Z}$ given by multiplications by complex numbers that are not real. The fractional linear action of ${\rm PSL}(2,\mathbb{Z})$ on $\mathcal{H}$ replaces $\mathcal{L}_\tau$ by a lattice $\lambda\mathcal{L}_\tau$, for a suitable complex number $\lambda\not=0$, which clearly also has complex multiplications.

Let $p, q$ be integers or infinity with $2\le p\le q\le\infty$ and $p^{-1}+q^{-1}<1$. A result of Takeuchi \cite{Tak1} says that there are exactly nine signatures $(p,q,\infty)$ such that any (Fuchsian) triangle group with that signature is arithmetic. A suitable conjugate in ${\rm PSL}(2,\mathbb{R})$ of any such arithmetic group is commensurable with ${\rm PSL}(2,\mathbb{Z})$, which has signature $(2,3,\infty)$. There are, clearly, infinitely many signatures $(p,q,\infty)$ such that triangle groups with that signature are non-arithmetic.

Let $f$ be a non-zero modular form for ${\rm PSL}(2,\mathbb{Z})$, with Fourier coefficients algebraic numbers. By  \cite{GMR1}, Theorem 1, any zero of $f$ is either transcendental or imaginary quadratic. We extend this result in \S\ref{s:mainresults}, Theorem \ref{t:zeros}, to zeros of non-zero automorphic forms for any Fuchsian triangle group with a cusp. In the non-arithmetic case, we need to replace the usual Fourier expansion by an expansion in another natural parameter at the cusp at $i\infty.$ We also need a suitable generalization of the notion of ``CM point," see \S\ref{s:modularembedding}. Overall, the modular embedding techniques of \cite{CoWo1} are what allow us to relax the assumption that the triangle group is arithmetic. Our result on zeros of automorphic forms seems to be the first of its nature for non-arithmetic groups.

\section{Automorphic Forms and \emph{Hauptmodulen}}\label{s:automorphic}

We begin with some remarks, some of which are classical facts, whereas others have their origin in \cite{DGMS} and the older reference \cite{WoAnal}. We will not directly quote \cite{WoAnal}, but refer the reader to \cite{DGMS} for appropriate acknowledgements. The Fuchsian groups considered from now on are triangle groups with cusps. Assume that $p$, $q$ are either integers or infinity, with $2\le p\le q$ and $\frac1p+\frac1q < 1$. A \emph{(Fuchsian) triangle group} with signature $(p,q,\infty)$ is by definition a subgroup of ${\rm PSL}(2,\mathbb{R})$ with presentation in terms of generators and relations given by:
$$
\langle M_1, M_2, M_3 : M_1^p=M_2^q=M_1M_2M_3=1\rangle.
$$
Let $\mathcal{H}$ be the upper half plane, that is the set of complex numbers $\tau$ with positive imaginary part. A Fuchsian triangle group acts on $\mathcal{H}$ by fractional linear transformations. The presentation determines the group up to conjugacy, and such a group has a fundamental region whose closure consists of two contiguous hyperbolic triangles with vertex angles $\frac{\pi}p$, $\frac{\pi}q$, $0$ (here, we take $\frac1\infty=0$). We fix the triangle group by specifying that its fundamental region must contain the triangle with vertices $\zeta_1=-\exp(-\pi i/p)$, $\zeta_2=\exp(\pi i/q)$, $\zeta_3=i\infty$ (again $\frac1\infty=0$). Let $\Gamma=\Gamma(p,q,\infty)$ be this triangle group with $\zeta_i$ the fixed point of $M_i$, $i=1,2,3$. The transformation $M_3$ is the translation $\tau\mapsto \tau+h_3$, where
$$
h_3=2\cos(\pi/p)+2\cos(\pi/q).
$$ 
The number $h_3$ is called the cusp width (of the cusp $\zeta_3$).

Let $f=f(\tau)$ be a meromorphic function on $\mathcal{H}$, and let $k$ be a positive \emph{even} integer. For $\gamma$ in ${\rm PSL}(2,\mathbb{R})={\rm SL}(2,\mathbb{R})/\{\pm{\rm Id}_2\}$, with representative $\begin{pmatrix}{a&b\cr c&d}\end{pmatrix}$ in ${\rm SL}(2,\mathbb{R})$, define 
\begin{equation}\label{e:act}
(f\mid_k\gamma)(\tau): =(c\tau +d)^{-k}f\left(\frac{a\tau+b}{c\tau+d}\right),\qquad {\rm for\;all}\;\tau\in{\mathcal H}.
\end{equation}
A \emph{holomorphic automorphic form of weight $k>0$ with respect to $\Gamma(p,q,\infty)$} is defined as follows. It is a function $f$ on $\mathcal{H}$ such that 
\begin{equation}\label{e:autk}
(f\mid_k\gamma)(\tau)=f(\tau),\qquad {\rm for\;all}\quad \tau\in\mathcal{H}\quad{\rm and\;all} \quad \gamma\in\Gamma(p,q,\infty),
\end{equation}
which is holomorphic on ${\mathcal H}$ and at \emph{all} the cusps of $\Gamma(p,q,\infty)$.  See \cite{DGMS}, \S2.1, for the definition of meromorphic and of holomorphic at a cusp. As $f$ is invariant under the transformation $\tau\mapsto \tau+h_3$, $\tau\in\mathcal{H}$, it has a Fourier expansion of the form
\begin{equation}\label{e:f1}
f(\tau)=\sum_{n\in\mathbb{Z}}a_n\exp(2\pi in\tau/h_3).
\end{equation}
As it is holomorphic at $\zeta_3=i\infty$, we have $a_n=0$ for all negative integers $n$. If we only require $f$ to be meromorphic at $\zeta_3=i\infty$, we allow it to have finitely many non-zero Fourier coefficients $a_n$ with $n$ negative. 

Let $f$ be holomorphic on ${\mathcal H}$, meromorphic at the cusps of $\Gamma(p,q,\infty)$, and invariant under the action of $\Gamma(p,q,\infty)$ by fractional linear transformations, so that $k=0$ in (\ref{e:autk}). We call $f$ an automorphic function. The automorphic functions for $\Gamma(p,q,\infty)$ form a field generated by any one such function, called a \emph{Hauptmodul}.  A \emph{Hauptmodul} for ${\rm PSL}(2,\mathbb{Z})$, which has signature $(2,3,\infty)$, is given by the classical elliptic modular function $j$. The choice of \emph{Hauptmodul} is unique once we specify a suitable normalization. Following \cite{DGMS}, we choose the \emph{Hauptmodul} $J_\Gamma$, $\Gamma=\Gamma(p,q,\infty)$, uniquely determined by the conditions
$$
J_\Gamma(\zeta_1)=1,\qquad J_\Gamma(\zeta_2)=0,\qquad J_\Gamma(\zeta_3)=\infty.
$$
(For $\Gamma(2,3,\infty)$, this gives $j/1728$.)
Wolfart \cite{WoAnal} was the first to investigate arithmetic properties of expansions around elliptic fixed points, and Fourier expansions about cusps, of automorphic forms and functions with respect to \emph{all} Fuchsian triangle groups. We focus on Fourier expansions around $\zeta_3=i\infty$. Wolfart proved that certain transcendence properties of the Fourier coefficients $a_n$ in (\ref{e:f1}) reflect whether or not the group $\Gamma(p,q,\infty)$ is arithmetic. Namely, the number $a_n$ is algebraic if and only if this group is arithmetic. More precisely, the number $a_n$ is of the form $b_n\alpha^n$, with $b_n$ algebraic and $\alpha$ transcendental, if it is non-arithmetic, see \cite{DGMS}. Let $\alpha_3$ be given by the closed formula in \cite{DGMS}, Theorem 1, see also \cite{WoAnal}. Then $\alpha_3$ is algebraic if $\Gamma(p,q,\infty)$ is arithmetic, and $\alpha_3/\alpha$ is algebraic if $\Gamma(p,q,\infty)$ is non-arithmetic. Let ${\widetilde q}_3=\alpha_3\exp(2 \pi i\tau/h_3)$. Again by  \cite{DGMS}, Theorem 1, the Fourier expansion of $J_\Gamma$ about $\zeta_3=i\infty$ is of the form
$$
{\widetilde q}_3^{\;-1}+\sum_{k=0}^\infty c_k{\widetilde q}_3^{\;k},
$$
with the $c_k$ polynomials (with signature independent coefficients) in $\mathbb{Q}[\frac1p, \frac1q]$. In particular, the above Fourier expansion lies in ${\widetilde q}_3^{\;-1}+\overline{\mathbb{Q}}[\![{\widetilde q}_3]\!]$. 

Let $L$ be the least common multiple of $p$ and $q$ if both these integers are finite. If $p$ is finite and $q=\infty$, let $L$ equal $p$, and if both $p$ and $q$ are infinite, let $L=1$. By \cite{DGMS}, Theorem 2, there is a closed formula for a holomorphic automorphic form $\Delta_{2L}$ of weight $2L$ which is the analogue of the discriminant form of weight $12=2{\rm lcm}(2,3)$ for the classical modular group ${\rm PSL}(2,\mathbb{Z})$ of signature $(p,q,\infty)=(2,3,\infty)$.  By \cite{DGMS}, Theorem 2, equation (2.17), we have
$$
\Delta_{2L}(\tau)=(-1)^L\left({\widetilde{q}}_3\frac{d}{d{\widetilde{q}}_3}J_\Gamma\right)^L\left(J_\Gamma\right)^{{\lceil\frac{L}{q}\rceil}-L}\left(J_\Gamma-1\right)^{{\lceil\frac{L}{p}\rceil}-L}
$$
\begin{equation}\label{e:delta}
={\widetilde{q}}_3^{\;{n_\Delta}}+O\left({\widetilde{q}}_3^{\;{n_\Delta + 1}}\right),
\end{equation}
where the Fourier expansion on the right hand side of the above equation has coefficients rational numbers and $n_\Delta=L(1-\frac1p-\frac1q)$. The automorphic form $\Delta_{2L}$ has no zeros in $\mathcal{H}$ nor at the $\Gamma$-orbits of the cusps not in the orbit $\Gamma(i\infty)$.

\section{Modular Embedding}\label{s:modularembedding}

To state our main results in \S\ref{s:mainresults}, we need the analog, for arbitrary signature $(p,q,\infty)$, of the classical notion of a complex multiplication (CM) point for the signature $(2,3,\infty)$. For this, and for the proofs of our results in \S\ref{s:mainresults}, we use the \emph{modular embedding} techniques of \cite{CoWo1}. For the nine arithmetic triangle groups of signature $(p,q,\infty)$, the CM points we define shortly all correspond to abelian varieties that are isogeneous to powers of a single elliptic curve with complex multiplication. This fact is not surprising, as these arithmetic groups are all commensurable with a subgroup of a triangle group with signature $(2,3,\infty)$. They are called \emph{modular triangle groups} in \cite{DGMS}. 

By the techniques of \cite{CoWo1}, for all $\Gamma(p,q,\infty)$ there is a positive integer $g$, a totally real number field $\mathbb{F}$ of degree $g$ over $\mathbb{Q}$, and an order $\mathcal{O}$ in $\mathbb{F}$ with the following property. There is a complex analytic embedding
$$
\varphi:\mathcal{H}\hookrightarrow{\mathcal{H}}^g
$$
and a group embedding
$$
\iota:\Gamma(p,q,\infty)\hookrightarrow\Gamma_{\mathbb{F}}
$$
such that 
$$
\varphi(\gamma(\tau))=\iota(\gamma)\varphi(\tau),\qquad {\rm for\;all}\quad \tau\in{\mathcal{H}},\quad \gamma\in \Gamma(p,q,\infty).
$$
Here, the group $\Gamma_{\mathbb{F}}$ is a suitable subgroup of ${\rm PSL}(2,\mathbb{F})$ commensurable with ${\rm PSL}(2,\mathcal{O})$. We can, and do, assume that $\varphi_1(\tau)=\tau$, for all $\tau\in\mathcal{H}$, where $\pi_i$ is the projection onto the $i$th factor of ${\mathcal H}^g$, and $\varphi_i(\tau):=(\pi_i\circ\varphi)(\tau)$, $i=1,\ldots,g$. We call $\varphi$ a \emph{modular embedding}, as the quotient space $\Gamma_{\mathbb{F}}\backslash{\mathcal{H}}^g$ is the coarse moduli space for an analytic family of complex abelian varieties of dimension $g$ whose endomorphism algebra contains a subfield isomorphic to $\mathbb{F}$, see \cite{Shi1}. 

We call $z=(z_1,\ldots,z_g)\in\mathcal{H}^g$ a \emph{CM point} if the $\Gamma_{\mathbb{F}}$-orbit of $z$ corresponds to an isomorphism class of abelian varieties with CM. Recall that an abelian variety $B$ of dimension $n$ has \emph{CM (complex multiplication)} if and only if there is a number field $\mathbb{L}$ contained in its endomorphism algebra ${\rm End}_0(B)={\rm End}(B)\otimes \mathbb{Q}$ with $[\mathbb{L}:\mathbb{Q}]=2n$. The field $\mathbb{L}$ is then a CM field, that is, a totally imaginary quadratic extension of a totally real number field. We call $\tau\in{\mathcal{H}}$ a \emph{CM point} if $\varphi(\tau)$ is a CM point. 

The analytic space $\Gamma_{\mathbb{F}}\backslash{\mathcal{H}}^g$ is isomorphic to the complex points $V(\mathbb{C})$ of a quasi-projective variety $V$ defined over $\overline{\mathbb{Q}}$, called a Hilbert modular variety. We assume that $V$ is embedded into a suitable projective space. There is a $\Gamma_{\mathbb{F}}$-invariant holomorphic map $J:{\mathcal{H}}^g\rightarrow V(\mathbb{C})$ that takes values in $V({\overline{\mathbb{Q}}})$ at all complex multiplication points of $\mathcal{H}^g$. 

From the joint results of \cite{Co1} and \cite{SW1} (see, for example Proposition 4, \S4, of the latter reference), we have the following criterion for a point of ${\mathcal H}^g$ to be CM.

\begin{proposition}\label{p:CMcharac} A point $z=(z_1,\ldots,z_g)\in\mathcal{H}^g$ is a CM point if and only if we have \emph{\bf both} $z_i\in{\overline{\mathbb{Q}}}$ for \emph{\bf {at least one}} $i=1,\ldots,g$ \emph{\bf and} $J(z)\in V(\overline{\mathbb Q})$.
\end{proposition}

 Applying, from  \cite{CoWo1}, the Theorem in \S2, the construction of \S3, Part 3, (see also Proposition 1, \S4), we can choose projective coordinates for $J(\varphi(\tau))$, $\tau\in{\mathcal H}$, that are polynomials, not all constant and with algebraic coefficients, evaluated at $J_\Gamma(\tau)$. 
 
 The discussion of this section yields the following characterization of CM points $\tau\in{\mathcal H}$. It generalizes the well-known result of Th. Schneider \cite{Sch1} on the values of the elliptic modular function.
 
 \begin{proposition} \label{p:schneider}
The point $\tau\in\mathcal{H}$ is a CM point if and only if we have both $\tau\in{\overline{\mathbb{Q}}}$ and $J_{\Gamma}(\tau)\in{\overline{\mathbb{Q}}}$. 
 \end{proposition}
 
 \begin{proof}  Let $\tau\in\mathcal{H}$ be a CM point. Then, by well-known facts from the theory of complex multiplication, see \cite{ShTa}, we have $\varphi_i(\tau)\in{\overline{\mathbb Q}}$, for all $i=1,\ldots,g$. Therefore $\varphi_1(\tau)=\tau\in{\overline{\mathbb Q}}$. We also have $J(\varphi(\tau))\in V(\overline{\mathbb{Q}})$. By the comments preceding this proposition, we deduce that $P(J_\Gamma(\tau))\in{\overline{\mathbb{Q}}}$ for some non-constant polynomial $P(x)\in{\overline{\mathbb{Q}}}[x]$. It follows that $J_\Gamma(\tau)\in{\overline{\mathbb{Q}}}$.

Conversely, suppose $\tau=\varphi_1(\tau)\in{\overline{\mathbb{Q}}}$ and $J_{\Gamma}(\tau)\in{\overline{\mathbb{Q}}}$. By the Theorem in \cite{CoWo1}, \S2, we have $J(\varphi(\tau))\in V(\overline{{\mathbb Q}})$. From Proposition \ref{p:CMcharac}, we deduce that $\varphi(\tau)$ is a CM point. Therefore, by definition, the point $\tau$ is CM.
 \end{proof}
 
 \section{Statement and Proof of Main Results}\label{s:mainresults}
 
 We can now state and prove our main results. They generalize to arbitrary triangle groups with cusps the results for $\Gamma(2,3,\infty)$ given by Theorem 1 and Theorem 2 of \cite{GMR1}.
 
 \begin{theorem}\label{t:zeros}
Let $2\le p\le q\le \infty$ with $\frac1p+\frac1q<1$. Let $f$ be a non-zero holomorphic automorphic form of weight $k>0$ with respect to $\Gamma(p,q,\infty)$ whose Fourier expansion lies in the ring ${\overline{\mathbb Q}}[\![\widetilde{q}_3]\!]$. Then any zero of $f$ is either a CM point or is a transcendental number.
 \end{theorem}
 
\begin{proof} Let $f$ be a non-zero holomorphic automorphic form of weight $k>0$ with respect to $\Gamma(p,q,\infty)$ whose Fourier expansion lies in the ring ${\overline{\mathbb Q}}[\![\widetilde{q}_3]\!]$. We defined $\Delta_{2L}$ in (\ref{e:delta}). Let $g$ be the function defined by
$$
g(\tau)=\frac{f^{2L}(\tau)}{\Delta_{2L}^k(\tau)},\qquad \tau\in{\mathcal H}.
$$ 
Then $g$ is an automorphic function (weight $0$) with no poles in $\mathcal{H}$, nor at the $\Gamma(p,q,\infty)$-orbits of the cusps not in $\Gamma(p,q,\infty)(i\infty)$. The function $g$ has a Fourier expansion in powers of $\widetilde{q}_3$ (where we allow negative integer powers) and with coefficients in ${\overline{\mathbb Q}}$. Therefore $g=P(J_\Gamma(\tau))$ for some $P(x)\in{\overline\mathbb{Q}}[x]$ which is not identically zero. If $\alpha$ is a zero of $f$, we have $g(\alpha)=P(J_\Gamma(\alpha))=0$, and it follows that $J_\Gamma(\alpha)$ is an algebraic number. From Proposition \ref{p:schneider}, it follows that $\alpha$ is either a CM point or is a transcendental number, as required.
\end{proof}

Following \cite{GMR1}, we define two automorphic forms $f_1$, $f_2$ with respect to the group $\Gamma(p,q,\infty)$ to be equivalent, denoted $f_1\sim f_2$, if there are natural numbers $k_1$ and $k_2$ such that we have the equality of functions $f^{k_1}=\lambda f^{k_2}$, for some $\lambda\in{\overline{\mathbb{Q}}}$, $\lambda\not=0$. We have the following generalization of \cite{GMR1}, Theorem 2.

\begin{theorem}\label{t:foverdelta} Let $f$ be a non-zero holomorphic automorphic form of weight $k>0$ with respect to $\Gamma(p,q,\infty)$ whose Fourier expansion lies in the ring ${\overline{\mathbb Q}}[\![\widetilde{q}_3]\!]$. Suppose that $f$ is not equivalent to $\Delta_{2L}$. If $\alpha\in\mathcal{H}\cap{\overline{\mathbb Q}}$ and $f^{2L}(\alpha)/\Delta_{2L}^k(\alpha)\in{\overline{\mathbb Q}}$, then $\alpha$ is a CM point.
\end{theorem}

\begin{proof}  As in the proof of Theorem \ref{t:zeros}, we deduce that $f^{2L}(\alpha)/\Delta_{2L}^k(\alpha)$ is a non-constant polynomial in $J_\Gamma(\alpha)$ with algebraic coefficients and therefore that $J_\Gamma(\alpha)$ is algebraic. By Proposition \ref{p:schneider}, we deduce that $\alpha$ is a CM point.
\end{proof}

\noindent \scriptsize{Department of Mathematics, Texas A\&M University, College Station, TX 77843-3368, USA; CNRS, UMR 8524, Universit\'e de Lille 1, Cit\'e Scientifique, 59655 Villeneuve d'Ascq, France

\noindent \emph{email:} paulatretkoff@tamu.edu}

\end{document}